\documentclass[english]{article}

\usepackage{preamble}
\usepackage{mymacros}

\begin{document}

\title{Homotopy Cardinality via Extrapolation of Morava-Euler Characteristics}
\author{Lior Yanovski\thanks{Einstein Institute of Mathematics, Hebrew University of Jerusalem.}}

\maketitle

\begin{abstract}
    We answer a question of John Baez, on the relationship between the classical Euler characteristic and the Baez-Dolan homotopy cardinality, by constructing a unique additive common generalization after restriction to an odd prime $p$. This is achieved by $\ell$-adically extrapolating to height $n=-1$ the sequence of Euler characteristics associated with the Morava $K(n)$ cohomology theories for (any) $\ell \mid p-1$. We compute this sequence explicitly in several cases and incorporate in the theory some folklore heuristic comparisons between the Euler characteristic and the homotopy cardinality involving summation of divergent series. 
\end{abstract}

\begin{figure}[H]
    \centering{}
    \includegraphics[scale=0.148]{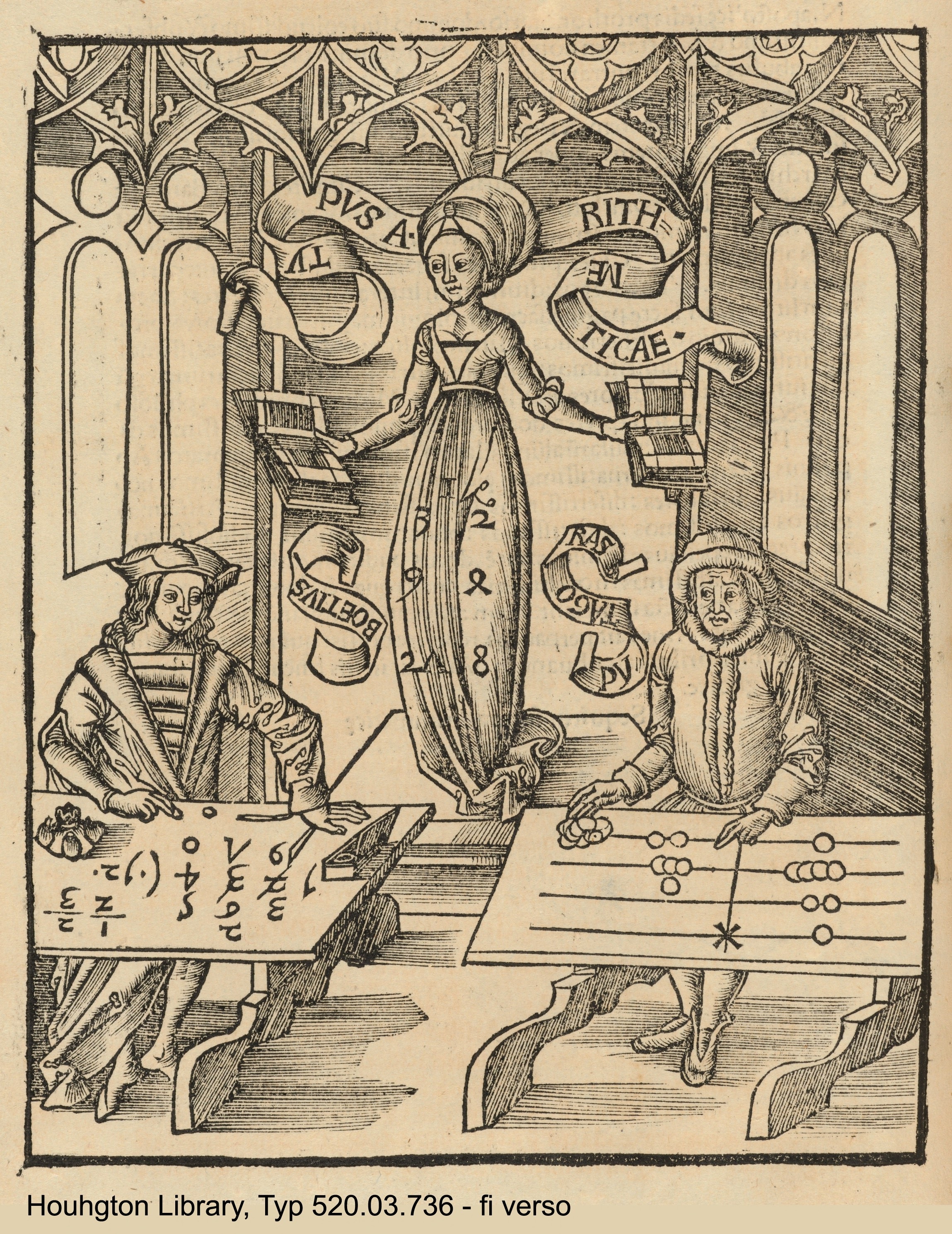}
    \caption*{Illustration from \textit{Margarita philosophica}, 1503, by Gregor Reisch (Typ 520.03.736, Houghton Library, Harvard University).}
\end{figure}

\tableofcontents

\section{Introduction}

\subsubsection{Counting and homotopy}

The original sin of decategorification is the passage from the category of finite sets $\Fin$ to its set of isomorphism classes $\NN$. The process of taking a finite set to its isomorphism class, i.e., its cardinality, forgets crucial information about the symmetries of the set as well as the non-invertible maps between different sets. However, it allows for a useful technique known as \textit{counting}. The categorical constructions of the binary co-product and product  in $\Fin$ are reflected, respectively, by the addition and multiplication binary operations on the set $\NN$. The cardinality map
\[
    |-| \colon \Fin \too \NN 
\]
thus becomes additive and multiplicative in the following sense:
\[
    |X \sqcup Y| = |X| + |Y| \qquad,\qquad |X\times Y| = |X|\cdot|Y|.
\]
This is the foundation of arithmetic.

The additivity property can be further extended to partially account for \textit{subtraction} as well. Given injections $X \hookleftarrow Z \into Y$ of finite sets, we have the inclusion-exclusion property
\[
    |X\cup_Z Y| =|X| + |Y| - |Z|. 
\]
The restriction to injections is an artifact of the discrete nature of $\Fin$. Each point of $Z$ identifies a point of $X$ with a point of $Y$ in the pushout $X\cup_Z Y$, but identifying the same two elements twice does not have any additional effect on the resulting set nor its cardinality. However, if we consider instead of the ordinary pushout the \textit{homotopy} pushout, each such identification is replaced by adding a \textit{path} between the two points.  Assigning to each such path the value $-1$ recovers the total value $|X| + |Y| - |Z|$ for the supposed cardinality of the pushout. Namely, the inclusion-exclusion property continues to hold if we take into account the number of times we have identified each two points. Going further, identifying two paths should be replaced by a ``path between paths'', or a 2-dimensional disc. Each such disc should contribute $+1$ to the total cardinality as it ``identifies two identifications'', and so on and so forth. This idea is formalized in the notion of \textit{Euler characteristic}. Let $\Spc^\fin$ be the category of finite spaces, i.e., those generated by the point under finite homotopy colimits. There exist a unique homotopy invariant function 
\[
    \chi \colon \Spc^\fin \too \ZZ,
\]
normalized by $\chi(\es) = 0$ and $\chi(\pt) = 1$, which is additive in the sense that 
\[
    \chi(X\cup_Z^h Y) = \chi(X) + \chi(Y) - \chi(Z).
\]

A presentation of a space $X$ as an (iterated) finite homotopy colimit of points equips it with a finite cell structure. If $x_n$ is the number of $n$-dimensional cells, then by the iterated inclusion-exclusion property we get
\[
    \chi(X) = \sum_{n \ge 0} (-1)^n x_n \qin \ZZ,
\]
so that such $\chi$ is unique and is computed in the expected manner. To show that this expression is well defined (i.e., depends only on the homotopy type of $X$) we can use the fact that it identifies with a manifestly homotopy invariant one via (co)homology:
\[
    \chi(X) = 
    \dim_\QQ H^{\mathrm{even}}(X;\QQ) - 
    \dim_\QQ H^{\mathrm{odd}}(X;\QQ).
\]
We can thus consider the Euler characteristic as the unique  additive extension of the notion of cardinality of finite sets to the world of homotopically finite spaces. 

Going back to ordinary cardinality of finite sets, we observe that the multiplicativity property can also be extended to partially account for \textit{division}. Let $G$ be a finite group acting freely on a finite set $X$. Then, we have
\[
    |X/G| = |X| / |G|,
\]
where $X/G$ is the set of orbits under the action. Here again, the restriction to free actions is dictated by the discreteness of finite sets. 
Considering the \textit{homotopy} orbits instead yields a $1$-truncated space (equivalently, a groupoid)
\[
    X/\!\!/G = \bigsqcup_{x \in X/G} BG_x.
\]
If we define for a finite group $G$ the cardinality of its classifying space $BG = \pt /\!\!/G$ to be $1/|G|$, we get by the orbit-stabilizer theorem
\[
    |X/\!\!/ G| = |X|/|G|.
\]
Intuitively, $BG$ has one point but with $G$ symmetries and we want to take these symmetries into account when considering its cardinality. Motivated by this, Baez and Dolan introduced in \cite{baez2001finite} the following higher homotopical generalization: 
\begin{defn}[Baez-Dolan]
    A space $X$ is called \textit{$\pi$-finite} if it has finitely many homotopy classes. Namely, there are finitely many connected components, each having finitely many non-vanishing homotopy groups and all of them are finite. For a connected $\pi$-finite space $X$ the \textit{homotopy cardinality} of $X$ is given by
    \[
        |X| = \prod |\pi_n(X)|^{(-1)^n} \qin \QQ_{\ge 0}.
    \]    
    We extend this function to non-connected $\pi$-finite spaces $X$ by summing the cardinalities of the (finitely many) components of $X$.
\end{defn}

The idea behind this definition is that elements of $\pi_2$ are symmetries of elements in $\pi_1$ and thus are ``symmetries of symmtries'', which leads to the above formula by a multiplicative analogue of the iterated inclusion-exclusion principle discussed above in relation to the Euler characteristic.  Using the long exact sequence in homotopy groups we easily deduce the extended multiplicative property of the homotopy cardinality: Given a fiber sequence of $\pi$-finite spaces
\[
    X \too Y \too Z,
\]
where $Z$ is connected, we have
\[
    |Y| = |X| |Z| \qin \QQ_{\ge 0}.
\]
Put differently, the $\pi$-finite group $G = \Omega Z$ acts on the $\pi$-finite space $X$ with homotopy orbits $X/\!\!/G = Y$, and we have
\[
    |X/\!\!/G| = |X||BG| = |X| / |G|.
\]
In particular, the Baez-Dolan homotopy cardinality
\[
    |-| \colon \Spc^{\pi\text{-}\mathrm{fin}} \too \QQ_{\ge 0}
\]
is the unique homotopy invariant \textit{multiplicative} extension of the cardinality of finite sets to suitably multiplicatively homotopy finite spaces. 

It is natural to ask whether the two homotopy theoretic extensions of counting, the Euler characteristic and the homotopy cardinality, agree. But, as noted already by Baez and Dolan, there is no way to compare the two as the required finiteness conditions are almost mutually exclusive. That is, the only finite and $\pi$-finite spaces are the finite sets. Nevertheless, an impressive number of heuristics collected in \cite{baez2005mysteries} hints that the two notions might be related after all. The most basic of such heuristics concerns the classifying space of a finite group $G$. The homotopy cardinality of $BG$ is by definition $1/|G|$, but since $BG$ is always infinite dimensional, it has no well-defined Euler characteristic. Still, the standard cell structure on $BG$ arising from the bar construction has $(|G|-1)^n$ cells in each dimension $n$. Thus, by applying naively the formula for the Euler characteristic we get the \textit{divergent} geometric series
\[
    \chi(BG) = \sum_{n=0}^\infty (-1)^n(|G| - 1)^n.
\]
Summing it forcibly, using any reasonable kind of regularization, yields the ``expected'' answer
\[
    \frac{1}{1+(|G| - 1)} = \frac{1}{|G|}.
\]
This and other heuristic manipulations with divergent series lead Baez to pose the following 

\begin{qst}[\cite{baez2003euler}]
    How are the Euler characteristic and homotopy cardinality related?
\end{qst}

There is a considerable body of literature on various generalizations of Euler characteristic, some of which are related to the above question (e.g. \cite{leinster2008euler}). More specifically, the subject of ``Euler characteristics of groups'' is quite and old one, and an excellent textbook account can be found in \cite[Chapter IX]{brown2012cohomology}. From our perspective, this theory can be seen as giving a well behaved common generalization of the Euler characteristic of finite spaces and the homotopy cardinality of classifying spaces of finite groups. However, it seems fair to say that no satisfactory answer was given so far to Baez's question in the full generality of $\pi$-finite spaces.\footnote{A solution using a different technique was claimed in the preprint \cite{berman2018euler}, but apparently it has a mistake. The proposed approach is however still interesting and merits further investigation.}

\subsubsection{Generalized homotopy cardinality}

One way to approach Baez's question is to first consider the extent to which the two notions of counting are \textit{compatible}. Namely, does there exist a natural class of spaces including both finite and $\pi$-finite spaces on which we can define a well behaved ``generalized homotopy cardinality'' function extending both the Euler characteristic on the former and the Baez-Dolan homotopy cardinality on the latter. A rather minimal such setting would be to take the category of spaces $\Spc^{\mathrm{small}}$ generated by the $\pi$-finite ones under finite homotopy colimits, and ask whether there is an additive function 
\[
    |-| \colon \Spc^{\mathrm{small}} \too \QQ,
\]
in the sense that
\[
    |X \cup^h_Z Y| = |X| + |Y| - |Z|,
\]
which assigns to $\pi$-finite spaces their homotopy cardinality.
Unfortunately, this is easily seen to be impossible. For every two distinct primes $p$ and $q$, the homotopy pushout square
\[\begin{tikzcd}
	{BC_p\times BC_q} & {BC_q} \\
	{BC_p} & {BC_p*BC_q =\pt}
	\arrow[from=1-1, to=2-1]
	\arrow[from=1-1, to=1-2]
	\arrow[from=1-2, to=2-2]
	\arrow[from=2-1, to=2-2]
\end{tikzcd}\]

would imply
\[
     |BC_p| + |BC_q| - |BC_p\times BC_q)| = |\pt|,
\]
so that
\[
    \Big(1 - \frac{1}{p}\Big)\Big(1 - \frac{1}{q}\Big) = 0.
\]
Negotiating with the failure, one might hope the obstruction to the existence of such a function lies solely in the interaction between different primes, and indeed this is our main result regarding Baez's question. Let $\Spc^{\psmall}$ be the category of spaces generated under finite colimits by $\pi$-finite \textit{$p$-spaces}, i.e., those whose homotopy groups are $p$-groups. 

\begin{thmx}\label{Intro_Generalzied_Homotopy_Cardinality}
    For every prime $p\neq 2$, there exists a (necessarily unique) generalized homotopy cardinality function 
    \[
        |-| \colon \Spc^{\psmall} \too \QQ,
    \]
    such that,
    \begin{enumerate}
        \item If $X \in \Spc^{\psmall}$ is a connected $\pi$-finite $p$-space, then 
        \[
            |X| = \prod_{n \ge 1} |\pi_n(X)|^{(-1)^n}.
        \]

        \item For every diagram $X\leftarrow Z \to Y$ in $\Spc^{\psmall}$ we have
        \[
            |X \cup^h_Z Y| = |X| + |Y| - |Z|.
        \]

        \item For every $X,Y \in \Spc^{\psmall}$, we have
        \[
            |X \times Y| = |X|\cdot|Y|.
        \]
    \end{enumerate}    
\end{thmx}

Admittedly, this result leaves something to be desired (in addition to treating the case $p=2$ of course). A somewhat more natural setting, which treats subtraction and division on equal footing, would be to consider the category $\Spc^{p\text{-}\mathrm{bigger}}$ of spaces generated by a point under both finite and $\pi$-finite $p$-space shaped colimits, and ask for a function
\[
    \norm{-} \colon \Spc^{p\text{-}\mathrm{bigger}} \too \QQ
\]
that satisfies both the generalized additivity property 
\[
    \norm{X \cup^h_Z Y} = \norm{X} + \norm{Y} - \norm{Z}
\]
and the generalized multiplicativity property
\[
    \norm{X /\!\!/ G} = \norm{X}/|G|
\]
for every $\pi$-finite $p$-group $G$ acting on $X\in\Spc^{p\text{-}\mathrm{bigger}}$, where $|G|$ is the Baez-Dolan homotopy cardinality. Note that we have $\Spc^{\psmall} \sseq \Spc^{p\text{-}\mathrm{bigger}}$, and a function $\norm{-}$ with the above properties would restrict to the generalized homotopy cardinality of \Cref{Intro_Generalzied_Homotopy_Cardinality}.
However, it is not known to the author at this point how to construct such a function $\norm{-}$ (nor even whether $\Spc^{p\text{-}\mathrm{bigger}}$ is actually bigger than $\Spc^{\psmall}$).

\subsubsection{Morava-Euler characteristics}

While the definition of the Euler characteristic of a space $X$ in terms of counting cells makes sense only when $X$ has a cell structure with finitely many cells, the definition in terms of rational cohomology
\[
    \chi(X) := 
    \dim_\QQ H^{\mathrm{even}}(X;\QQ) - \dim_\QQ H^{\mathrm{odd}}(X;\QQ)
\]
makes sense for all spaces $X$ with finite dimensional rational cohomology, which includes all $\pi$-finite spaces. Moreover, $\chi$ is additive on all homotopy pushouts of spaces on which it is defined, so in particular, we do get an additive function
\[
    \chi \colon \Spc^{\psmall} \too \ZZ.
\]
The problem of course is that it doesn't give the desired values on $\pi$-finite $p$-spaces (i.e., their homotopy cardinality), if only because it assumes only \textit{integral} values. In fact, a $\pi$-finite space $X$ is discrete in the eyes of rational cohomology, so we get $\chi(X) = |\pi_0(X)|$. The key idea behind the construction of the generalized homotopy cardinality of \Cref{Intro_Generalzied_Homotopy_Cardinality} is to consider Euler characteristics with respect to other cohomology theories, to produce a sequence of additive functions 
\[
    \chi_n \colon \Spc^{\psmall} \too \ZZ.
\]
While non of them is going to be the desired generalized homotopy cardinality (since, again, they all assume integral values) a certain limiting process will yield a function
\[
    |-| \colon \Spc^{\psmall} \too \QQ
\]
with the correct values on $\pi$-finite $p$-spaces. Furthermore, the additivity in homotopy pushouts will persist through this limiting process, so it will be enjoyed by $|-|$ as well.

The cohomology theories alluded to above are the \textit{Morava $K$-theories} from chromatic homotopy theory. For an implicit fixed prime $p$ and a chromatic height $0\le n \le \infty$, we have the Morava $K$-theory ring spectrum $K(n)$ representing a certain (extraordinary) cohomology theory for spaces. At the extremal cases we have the ordinary cohomology theories
\[
    K(0)^* (X) = H^*(X; \QQ) \quad,\quad K(\infty)^* (X) = H^*(X; \FF_p),
\]
while the intermediate values of $n$ provide a certain extrapolation between zero characteristic and positive characteristic, which exists only in ``higher algebra''. For each $K(n)$, the coefficient ring $\kappa_n := K(n)^*(\pt)$ is a graded field, so we may consider the dimension of $K(n)^*(X)$ over $\kappa_n$. By the fundamental computation carried out in \cite{ravenel1980morava}, the objects of $\Spc^{\mathrm{small}}$ have finite dimensional $K(n)$-cohomology for every finite $n$. 

\begin{defn}
    For every $X \in \Spc^{\mathrm{small}}$, the \textit{Euler-Morava characteristic} of height $n$ of $X \in \Spc^{\mathrm{small}}$ is given by
    \[
        \chi_n(X) := 
        \dim_{\kappa_n} K(n)^{\mathrm{even}}(X) -
        \dim_{\kappa_n} K(n)^{\mathrm{odd}}(X).
    \] 
\end{defn}

This produces the desired sequence of additive functions when restricted to $\Spc^{\psmall}$. Remarkably, using a deep theory of \textit{tempered ambidexterity} developed by Lurie, these numbers can be described in elementary terms by counting homotopy classes of maps from torii.

\begin{thm}[{\cite[Corollary 4.8.6]{Lurie_Ell3}}]
    For a $\pi$-finite $p$-space $X$, 
    \[
        \chi_{n} (X) = |[\TT^n, X]| \qin \NN,
    \]
    where $\TT^n = (S^1)^n$ is the $n$-dimensional torus.
\end{thm}

This formula is quite remarkable. For example, note that while the left hand side is given by the \textit{difference} of two dimensions, the right hand side is given by the cardinality of a single functorially defined set.

\begin{rem}\label{rem:Char_Non_p}
    In fact, Lurie proves more generally that for any $\pi$-finite space $X$ the Morava-Euler characteristic $\chi_n(X)$ equals the number of homotopy classes of maps from the $p$-adic torus $B^n\ZZ_p$ to~$X$. In this paper we shall mostly stick to $p$-spaces where these are the same. The only exception is when we consider the example of $X = BG$ for a general finite group $G$ in \Cref{Sec:Examples}. This special case predates Lurie's \cite{Lurie_Ell3} and was already treated in the work of Hopkins-Kuhn-Ravenel on generalized character theory \cite{hopkins2000generalized}.
\end{rem}

Shifting perspective, we consider for a fixed $\pi$-finite $p$-space $X$, the function $\lfun_X \colon \NN \to \NN$ given by 
\[
    \lfun_X(n) := \chi_n(X) = |[\TT^n,X]|.
\]

Our main result regarding this function is the following:

\begin{thmx}\label{Intro_Morava_Euler_Cont}
    For every $\pi$-finite $p$-space $X$ and a prime $\ell \mid (p-1)$, the function $\lfun_X \colon \NN \to \NN \sseq \QQ_\ell$
    extends uniquely to a continuous $\ell$-adic function 
    \[
        \hat{\lfun}_X \colon \ZZ_\ell \too \QQ_\ell,
    \] 
    such that (independently of $\ell$)
    \[
        \hat{\lfun}_X(-1) = |X| \qin \QQ \ (\sseq \QQ_\ell). 
    \]
\end{thmx}

It follows by additivity that for every $X \in \Spc^{\psmall}$, the sequence of Morava Euler characteristics $\chi_n(X)$ is uniformly $\ell$-adically continuous in $n$ and its extrapolation to $n=-1$ is a rational number. We get that
\[
    |-| := \chi_{-1} \ \colon\ \Spc^{\psmall} \too \QQ
\]
is the desired generalized homotopy cardinality function, thus proving \Cref{Intro_Generalzied_Homotopy_Cardinality}. 

To get a feeling for the above construction of the generalized homotopy cardinality, it is instructive to consider the following two examples:
\begin{enumerate}
    \item For $X$ a finite space, the sequence $\chi_n(X)$ is \textit{constant}, so in particular 
    \[
        \chi_n(X) = \chi_0(X) \qin \ZZ.
    \] 
    Taking the $\ell$-adic limit $n\to -1$ we still get the classical  Euler characteristic of $X$.

    \item For $X = B^d C_p$ the Eilenberg-MacLane space for $C_p$ of  degree $d$, we have 
    \[
        \chi_n(B^dC_p) = p^{\binom{n}{d}} \qin \ZZ.
    \]
    Now, while the function $\binom{n}{d}$ is defined combinatorially only for $n \ge 0$, it is polynomial in $n$ and has therefore an extrapolation to $n = -1$ given by
    \[
        \binom{-1}{d} = 
        \frac{(-1)(-2)\cdots(-d)}{1\cdot 2 \cdots d} 
        = (-1)^d.
    \]
    Furthermore, since $p \equiv 1$ modulo $\ell$, the function $p^n$ is also $\ell$-adically continuous, so taking the $\ell$-adic limit $n \to -1$ of the composition $p^{\binom{n}{d}}$, we get
    \[
         p^{\binom{-1}{d}} = p^{(-1)^d} = |B^dC_p|.
    \]
\end{enumerate}

We note that for a general $\pi$-finite $p$-space $X$, the cohomology theory $K(n)$ sees only the $n$-th Postnikov truncation of $X$, so we really need the entire sequence of functions $\chi_n$ to be able to recover the homotopy cardinality of all $\pi$-finite $p$-spaces. A curious by-product of \Cref{Intro_Morava_Euler_Cont} is that the homotopy cardinality of such spaces is a \textit{stable} invariant. 


\subsubsection{Relation to higher semiadditivity}

Our proof of \Cref{Intro_Morava_Euler_Cont} is entirely elementary. We analyse the numerical properties of the sequence of numbers $\lfun_X(n) := |[\TT^n,X]|$ related to the Euler-Morava characteristics of $X$ only via Lurie's formula, which we treat as a black-box. Nevertheless, we take the opportunity to provide some theoretical context, that might point towards deeper structural phenomena of which the results of this paper are mere numerical shadows (and perhaps also to stronger versions thereof). As this is solely for expository purposes, we merely sketch the relevant ideas in an informal way. More information can be found in \cite{HLAmbiKn, TeleAmbi, AmbiHeight, Lurie_Ell3}. 

To begin with, a more highbrow way to define the classical Euler characteristic of a finite space $X$ is as the \textit{symmetric monoidal dimension} of its suspension spectrum $X\otimes \SS \in \Sp$. Namely, for a general dualizable spectrum $X$, its dimension (trace of the identity) is an element of $\pi_0(\SS) = \ZZ$.
Since the category $\Sp$ is \textit{stable}, the dualizable objects
in it are closed under finite (co)limits, and their dimension satisfies the additivity property
\[
    \dim(X \cup_Z^h Y) = \dim(X) + \dim(Y) - \dim(Z).
\]
We thus recover the classical Euler characteristic by restricting to suspension spectra of finite spaces (see \cite{ponto2014linearity} for more on this perspective).

Similarly, the deep reason for the finite dimensionality of $K(n)$-cohomology of $\pi$-finite spaces is the \textit{$\infty$-semiadditivity} of the categories $\Sp_{K(n)}$ of $K(n)$-local spectra in the sense of Hopkins-Lurie (\cite{HLAmbiKn}). This roughly means that the limit and colimit of a diagram in $\Sp_{K(n)}$ indexed over a $\pi$-finite space are canonically isomorphic. This remarkable property has many consequences, among them that the dualizable objects in $\Sp_{K(n)}$ are closed under $\pi$-finite (co)limits. Since $\Sp_{K(n)}$ is also stable, the dualizable objects are closed also under finite (co)limits. In particular, for every $X\in \Spc^{\psmall}$, the $K(n)$-localized suspension spectrum $X\ \widehat{\otimes}\ \SS_{K(n)}$ is dualizable in $\Sp_{K(n)}$, and we can compute its dimension
\[
    \dim(X\ \widehat{\otimes}\ \SS_{K(n)}) \in \pi_0(\SS_{K(n)}),
\]
and this dimension function is additive as before. The ring $\pi_0(\SS_{K(n)})$ is a nilpotent thickening of $\ZZ_p$, and while work of Carmeli and Yuan shows that exotic nilpotents \textit{can} show up in this way, ignoring the nilpotent part we get precisely the integer
\[
    \chi_n(X) := 
    \dim_{\kappa_n} K(n)^{\mathrm{even}}(X) -
    \dim_{\kappa_n} K(n)^{\mathrm{odd}}(X)
    \qin \ZZ \sseq \ZZ_p.
\] 
To summarize, the Euler-Morava characteristics of a space can be defined in terms of the symmetric monoidal dimension of dualizable objects in $\Sp_{K(n)}$, which are closed under $\pi$-finite colimits by the higher semiadditivity of these categories. The homotopy cardinality of Baez and Dolan turns out to be also closely related to higher semiadditivity. Just like in an ordinary semiadditive category morphisms can be added, in a higher semiadditive category morphisms can be integrated along $\pi$-finite spaces. In particular, integrating the identity morphism of an object over a $\pi$-finite space $X$, can be thought of as multiplication by the ``cardinality of $X$'', which is an element of the ring $\pi_0(\SS_{K(n)})$. We thus get a sequence of such cardinalities 
\[
    |X|_n \in \pi_0(\SS_{K(n)}).
\]
Moreover, the $K(n)$-local symmetric monoidal dimension of a $\pi$-finite $p$-space $X$ can be formulated in these terms by the formula
\[
    \dim (X \otimes \SS_{K(n)})) = |LX|_n.
\]

For height $n=0$, i.e. for the category of rational spectra $\Sp_{\QQ} = \Sp_{K(0)}$, this higher semiadditive notion of cardinality is precisely the Baez-Dolan homotopy cardinality
\[
    |X|_0 = |X| \in \QQ = \SS_{\QQ}.
\]
For higher heights, Lurie's machinery of tempered ambidexterity can be applied to show that the integral part of $|X|_n$ coincides with $|L^nX|_0$. The combination of all of the above puts the Euler-Morava characteristics and the Baez-Dolan homotopy cardinality on equal footing. Namely, the sequence 
\[
    |X|,\ \chi_0(X),\ \chi_1(X),\ \chi_2(X),\ \chi_3(X),\ \dots  
\]
coincides up to nilpotents with the sequence
\[
    |X|_0,\ \ |X|_1,\ \ |X|_2,\ \ |X|_3, \ \  |X|_4, \ \ \dots
\]

We can thus roughly interpret \Cref{Intro_Morava_Euler_Cont} as saying that the sequence of higher semiadditive cardinalities of a $\pi$-finite $p$-space $X$ is $\ell$-adically continuous in the chromatic height $n$.

\subsubsection{Outline}

In \Cref{Sec:Proofs}, we prove \Cref{Intro_Morava_Euler_Cont}, from which \Cref{Intro_Generalzied_Homotopy_Cardinality} follows as explained above. The proof divides naturally in two: showing that $\lfun_X(n)$ is $\ell$-adically continuous, and showing that the extrapolation to $n=-1$ is $|X|$. 

In \Cref{Sec:Examples}, we consider three families of examples. First, we show that for loop-spaces the function $\lfun_X(n)$ does not see the Postnikov invariants of $X$ and consequently has a very simple combinatorial formula. Second, we use the results of \cite{hopkins2000generalized} to analyse the case $X = BG$ where $G$ is a general finite (not necessarily $p$-)group. Third, we construct a certain family of $\pi$-finite $p$-spaces with non-trivial Postnikov invariants and give a formula for $\lfun_X(n)$ in terms of $p$-binomial coefficients.

In \Cref{Mysteries}, we return to the heuristic comparison of $|BG|$ and $\chi(BG)$. We reinterpret this in terms of $\ell$-adic continuity of the generalized homotopy cardinality, and set up a general framework of $\ell$-adically continuous resolutions extending this phenomena to higher $\pi$-finite groups.

\subsubsection{Acknowledgments}

Most of the results presented here were obtained during my first Ph.D project back in 2017 under the supervision of Tomer Schlank (who was the one to suggest to ``somehow use Morava $K$-theories''). they were subsequently announced in the traschromatic homotopy theory conference held in Regensburg that year and presented in several talks since then, but various distractions delayed the writing process. It should be noted that at the time, Lurie's \cite{Lurie_Ell3} was not published yet and much of the general higher semiadditive context was still missing as well. Thus, in a sense, the present time is more ripe for communicating these ideas, which may partially excuse the long delay. 

I am deeply grateful to Tomer Schlank for his help with this project, as well as his mentoring over the passing years. I also thank him for useful comments on an earlier draft of this paper. 
I wish to extend my gratitude also to Maxime Ramzi for reviving  my interest in this project through many stimulating conversations motivating me to finally write this paper. I would also like to thank John Baez for his numerous excellent expositions of these and other higher categorical ideas, that got me into thinking about this project in the first place. Finally, I would like to thank all the present and former members of the seminarak group for stimulating conversations and encouragement, with special thanks to Shai Keidar and Shaul Ragimov for providing the elegant argument in the proof of the second part of \Cref{Prob_Prime2p}.

\section{Morava-Euler Characteristics}\label{Sec:Proofs}

In this section we will prove \Cref{Intro_Morava_Euler_Cont}, from which \Cref{Intro_Generalzied_Homotopy_Cardinality} follows as we already explained above. 
That is, we need to show that given a $\pi$-finite $p$-space $X$, the function $\lfun_X(n) = |[\TT^n, X]|$ extends to an $\ell$-adically continuous function $\hat{\lfun} \colon \ZZ_\ell \to \QQ_\ell$ with $\hat{\lfun}_X(-1) = |X|.$ We shall apply the following criterion:

\begin{thm}[Mahler {\cite{mahler1958interpolation}}]
    A function $f\colon \NN \to \QQ_\ell$ extends to a (necessarily unique) continuous function $\hat{f}\colon \NN \to \QQ_\ell$, if and only if its inverse binomial transform 
    \[
        \cl{f}(n) = \sum_{k=0}^n (-1)^{n-k}\binom{n}{k} f(k)
    \]
    satisfies $\cl{f}(n) \to n$ in the $\ell$-adic norm when $n\to \infty$. Furthermore, in this case the value of $\hat{f}$ on $a\in \ZZ_\ell$ is given by 
    \[
        \hat{f}(a) = \sum_{k=0}^\infty \binom{a}{k} \cl{f}(k).
    \]
\end{thm}

Our strategy is to verify this criterion for $\lfun_X$ of a $\pi$-finite $p$-space $X$, using induction on the Postnikov tower of $X$ and an explicit analysis of the case $X = B^d C_p$. 

\subsubsection{Categorifying Mahler's theorem}

To carry out this strategy, we will need to ``categorify'' all the ingredients in Mahler's theorem. 

\begin{defn}
    An order preserving injection $[k] \into [n]$ induces a projection of tori $\TT^n \to \TT^k$. We call these the \textit{standard projections}.
    We say that an element of $L^n X$ is \textit{old} if it is induced from an element $\TT^k \to X$ by a standard projection $\TT^n \to \TT^k \to X$ for some $k<n$, and \textit{new} otherwise. We let 
    \[
        LX = L_\old X \ \sqcup\  L_\new X 
    \]    
    be the partition of $L X$ into old (i.e. constant) and new elements respectively. The new elements in $L^n X$ are thus $L^n_\new X \sseq L^n X$ and the old ones are the complement. 
\end{defn}

\begin{war}
    While we may interpret the notation for the space of new elements $L^n_\new X \sseq L^n X$ as the $n$-fold iterate of the functor $L_\new$ applied to $X$, the analogous claim is not true for the \textit{old} elements in $L^n X$. In particular, we reserve the notation $L^n_\old X$ for the the $n$-fold iterate of the functor $L_\old$ applied to $X$, which for $n\ge 1$ is usually a proper subspace of the space of old elements $L^n X \smallsetminus L_\new^n X$.
\end{war}

Note that the standard projections $\TT^n \to \TT^k$ admit \textit{sections}, so for every $X$, the corresponding maps $L^k X \to L^n X$ admit \textit{retracts}. In particular, the induced maps $\pi_0(L^k X) \to \pi_0 (L^n X)$ are \textit{injective}. Furthermore, we would like to show that each element of $\pi_0(L^nX)$ is induced from a unique new element in $\pi_0(L^kX)$ for a unique minimal subset of coordinates $[k] \into [n]$. Since we are dealing with factorizations up to homotopy, this is not completely obvious. We shall use the following:

\begin{lem}\label{lem:join}
    Let $A,B,C \in \Spc$ be spaces with a map
    \[
        f \colon A\times B \times C \too X.
    \]
    If $f$ factors up to homotopy through each of the projections 
    \[
    A \times C \longleftarrow 
    A\times B \times C \too  
    B \times C,
    \]
    then it factors up to homotopy through the projection $A\times B \times C \to C$.
\end{lem}
\begin{proof}
    Consider first the case $C = \pt$. If $f\colon A\times B \to X$ factors through both projections 
    \[
        A \longleftarrow A\times B \too B
    \] 
    up to homotopy, then it factors through their homotopy pushout given by the join:
    \[\begin{tikzcd}
    	{A \times B} & B \\
    	A & {A\star B}
    	\arrow[from=1-1, to=2-1]
    	\arrow[from=2-1, to=2-2]
    	\arrow[from=1-1, to=1-2]
    	\arrow[from=1-2, to=2-2]
    	\arrow["\lrcorner"{anchor=center, pos=0.125, rotate=180}, draw=none, from=2-2, to=1-1]
    \end{tikzcd}\]
    The claim now follows from the standard fact that the bottom and right maps are nullhomotopic (see, e.g., \cite[Proposition 2.21.]{devalapurkar2021james}). The claim for a general $C$ follows from the above special case by multiplying the pushout square defining $A\star B$ above with $C$ and observing that it is still a pushout square.
\end{proof}

Consequently,

\begin{prop}
    Let $X$ be a space and $n\ge 0$. Each element of $\pi_0(L^n X)$ is induced from a unique new element of $\pi_0(L^k X)$ along a unique inclusion $[k] \into [n]$.
\end{prop}

\begin{proof}
    By \Cref{lem:join}, if a map $\gamma \colon \TT^n \to X$ factors through two subsets of coordinates of $[n]$, then it factors through their intersection. Hence, there is a minimal $[k] \into [n]$ for which $\gamma$ is in the image of the restriction $\pi_0(L^k X) \to \pi_0(L^n X)$, hence new, and since this map is an inclusion, the pre-image is unique.
\end{proof}

We get that the collection of subsets  $\pi_0 (L_\new^n X) \sseq \pi_0 (L^n X)$ categorifies the binomial transform of the function $\lfun_X(n)=|\pi_0(L^nX)|$. 

\begin{cor}\label{fbar_new}
    For every $\pi$-finite $X$, we have $\cl{\lfun}_X(n) = |\pi_0 (L^n_\new X)|$ for all $n\ge 0$.
\end{cor}
\begin{proof}
    By the above, each element of $\pi_0(L^n X)$ is induced from a unique new element of $\pi_0(L^k X)$ along a unique $[k] \into [n]$.
    We thus get,
    \[
        \pi_0(L^n X) \simeq \bigcup_{k=0}^n \binom{[n]}{k} \times \pi_0(L_\new^k X),
    \]
    from which follows that
    \[
        \lfun_X(n) = |\pi_0(L^n x)| = \sum_{k=0}^n \binom{n}{k}|L_\new^k X|,
    \]
    and hence
    \[
        \cl{\lfun}_X(n) = \sum_{k=0}^n (-1)^{n-k}\lfun_X(n) = |L_\new^n X|. 
    \]
\end{proof}

We therefore need to show that the number of new maps $\TT^n \to X$ is divisible by a high power of $\ell$. We shall now categorify this property as well. Since $X$ is a $\pi$-finite $p$-space, we have
\[
    L^n X = \Map(B\ZZ^n , X) = \Map(B\ZZ_p^n , X).
\]

Since $\ell \mid (p-1)$ we have an embedding $\ZZ/\ell \into \FF_p^\times$, and $\FF_p^\times$ further embeds into $\ZZ_p^\times$ by Teichmuller lifts. Thus, the group $(\ZZ/\ell)^n$ acts on $L^n X$ via the embedding as diagonal matrices
\[
    (\ZZ/\ell)^n \into (\ZZ_p^\times)^n \le \GL_n(\ZZ_p).
\]
We let $G_n \le \GL_n(\ZZ_p)$ be the image of this embedding. 

\begin{lem}
    For every space $X$, the action of $G_n$ on $L^n X$ preserves $L^n_\new X \sseq L^n X$.
\end{lem}

\begin{proof}
    This follows immediately from the the fact that the standard projections are $G_n$-equivariant.
\end{proof}

Consequently, by the orbit-stabilizer theorem, it suffices to show that an element of $\pi_0(L^n_\new X)$ can't have a very large stabilizer $G \le G_n$. Namely, that the codimension $\dim (G_n/G)$ is bounded from below by a function of $n$ that grows to infinity. 

\subsubsection{$\ell$-Adic continuity}

To facilitate induction, we shall prove a sharper result, that identifies for every sufficiently large $G \le G_n$ a specific ``degenerate index'' $1\le i \le n$ for all elements of $\pi_0(L^n X)$ that are fixed by $G$, depending only on the truncation level of $X$. For this, we shall use the action of $G_n$ on the \textit{space} $L^n X$ and analyze its \textit{homotopy} fixed point. Note that since any subgroup $G\le G_n$ is an $\ell$-group and the homotopy groups of $L^n X$ are $p$-groups, we have a natural isomorphism
\[
    \pi_0((L^nX)^{hG}) \iso (\pi_0(L^n X))^G.
\]
We shall use this repeatedly.

\begin{defn}
    For $G \le G_n$ and $d\in \NN$, we say that an index $1\le i\le n$ is \textit{$d$-degenerate} under $G$, if $G$ is not contained in any subgroup of the form
    \[
        G_{\Vec{i}} = \{g \in G_n \mid g_{i_1} g_{i_2} \cdots g_{i_{d'}} = 1\} \le G_n,
    \]
    for any $d' \le d$ and a multi-index $\Vec{i}=(1\le i_1<i_2<\dots<i_{d'}\le n)$ containing $i$.
\end{defn}

\begin{rem}
    By definition, if $i$ is $d$-degenerate under $G\le G_n$, then it is also $(d-1)$-degenerate under it. Namely, the conditions are of increasing strength.     
\end{rem}

The main ingredient in proving the $\ell$-adic continuity of $\lfun_X$ is the following claim:

\begin{lem}\label{Covered_Old}
    Let $G \le G_n$ be a subgroup under which a certain index $1\le i\le n$ is $d$-degenerate. For every $d$-finite $p$-space $X$, the map induced from $[n\ssm i] \into [n]$ is an isomorphism
    \[
        (L^{n\ssm i}X)^{hG} \iso (L^n X)^{hG}.
    \]
\end{lem}

\begin{proof}
    We prove the claim by induction on $d$ and the total size of $\pi_* X$. For $d=0$, the claim is trivial, so we assume $d\ge 1$. By induction, we may assume that $\pi_d(X) \neq 0$ and we have a fiber sequence
    \[
        X \too Y \too B^{d+1}C_p
    \]
    with $Y$ a $d$-finite $p$-space with $|\pi_*Y| < |\pi_*X|$. Since both $L^k$ and $(-)^{hG}$ are limit preserving functors we get a map of fiber sequences:
    \[\begin{tikzcd}
    	{(L^{n\ssm i} X)^{hG}} && {(L^{n\ssm i} Y)^{hG}} && {(L^{n\ssm i} B^{d+1} C_p)^{hG}} \\
    	{(L^n X)^{hG}} && {(L^n Y)^{hG}} && {(L^n B^{d+1} C_p)^{hG}}
    	\arrow[from=2-1, to=2-3]
    	\arrow[from=2-3, to=2-5]
    	\arrow[from=1-1, to=1-3]
    	\arrow[from=1-3, to=1-5]
    	\arrow[from=1-3, to=2-3]
    	\arrow[from=1-1, to=2-1]
    	\arrow[from=1-5, to=2-5]
    \end{tikzcd}\]
    The middle vertical map is an isomorphism by induction, hence to show that the left vertical map is an isomorphism, it suffices to show that the right vertical map is an isomorphism. In fact, it suffices only to show that it is an inclusion of connected components, since we can discard the connected components not hit by the right horizontal maps. 
    Since the right vertical map admits a retract, it is injective on all homotopy groups and in particular on $\pi_0$. Moreover, it is a map of loop-spaces (i.e. group objects in spaces), so it suffices to show that is it an isomorphism on the connected component of the base-point, or in other words, after applying $\Omega$.
    Since everything commutes with $\Omega$, we get the analogous map
    \[
        (L^{n\ssm i} B^d C_p)^{hG} \too
        (L^n B^d C_p)^{hG}.
    \]    
    By induction, this map is an isomorphism on all homotopy groups in degrees $\ge 1$ and is also injective on $\pi_0$, as it admits a retract. It therefore remains to show that it is surjective on $\pi_0$. Observe that
    \[
        \pi_0 L^n B^d C_p \simeq 
        H^d(\TT^n ; \FF_p) = 
        \bigwedge^d {\FF_p}\{x_1,\dots,x_n\}.
    \]
    An element in this group is thus a linear combination of wedges 
    $x_{\Vec{i}} = x_{i_1}\wedge x_{i_2}\wedge \cdots\wedge x_{i_d}$. 
    The action of $g = (g_1,\dots,g_n) \in G_n$ on $x_{\Vec{i}}$ is by the obvious formula
    \[
        gx_{\vec{i}} = 
        (g_{i_1}x_{i_1})\wedge \cdots\wedge (g_{i_d}x_{i_d}) 
        = (g_{i_1}\cdots g_{i_d}) x_{\vec{i}}.
    \]
    Therefore, $x = \sum c_{\vec{i}}x_{\Vec{i}}$ is fixed by $g$ if and only if $g_{i_1}\cdots g_{i_d} = 1$ for every $\vec{i}$ with $c_{\vec{i}} \neq 0$. Since $i$ is $d$-degenerate under $G$, we get that a $G$-fixed element $x$ must have $c_{\Vec{i}} = 0$ for all $\Vec{i}$ containing $i$. In other words, it is in the image of 
    \[
        \pi_0 (L^{n\ssm i} B^d C_p) \simeq 
        H^d(\TT^{n\ssm i} ; \FF_p) = 
        \bigwedge^d {\FF_p}\{x_1,\dots,\widehat{x_i},\dots,x_n\}.
    \]
\end{proof}

To make use of the above, we also need the following numerical estimate for the size of the subgroup $G\le G_n$, that ensures the existence of $d$-degenerate index.

\begin{lem}\label{Codim_Degenerate}
    For $G\le G_n$ such that $\dim(G_n/G) < \frac{n}{d}$, there exists an index $1\le i \le n$ which is $d$-degenerate under $G$.
\end{lem}
\begin{proof}
    Assume that for all $j = 1,\dots n$, the index $i$ is \textit{not} $d$-degenerate under $G$, so there exists a corresponding $d'$-tuple $\Vec{i}_j$ with $G \sseq G_{\Vec{i}_j}$. Consequently,
    \[
        G \sseq G_{\Vec{i}_1} \cap \cdots \cap G_{\Vec{i}_n}.
    \]
    Namely, the elements of $G$ satisfy a system of $n$ linear equations involving together all the $n$ variables (since the $j$-th equation involves the $j$-th variable). Since each of these linear equations involves only $\le d$ variables, there is a subsystem of $\floor{\frac{n}{d}}$ equations
    on pairwise disjoint subsets of variables, which are therefore linearly independent. Consequently, the rank of the system of equations is at least $\frac{n}{d}$, which implies that the co-dimension of $G$ is at least $\frac{n}{d}$.
\end{proof}

Combining the above results we get a bound from below on the sizes of the $G_n$-orbits in $\pi_0(L_\new^n X)$.

\begin{prop}\label{Orbit_ell}
    For every $d$-finite $p$-space $X$, the size of each $G_n$-orbit in $\pi_0(L_\new^n X)$ is divisible by $\ell^{\floor{n/d}}$.
\end{prop}
\begin{proof}
    Let $\gamma \in \pi_0(L^n X)$ be an element with stabilizer $G_\gamma \le G_n$ of co-dimension $< \floor{n/d}$. By \Cref{Codim_Degenerate}, there is an index $1\le i\le n$, such that $i$ is $d$-degenerate under $G_\gamma$. Consequently, by \Cref{Covered_Old}, we have 
    \[
        (L^{n\ssm i}X)^{hG_\gamma} \iso (L^n X)^{hG_\gamma}.
    \]
    Using the fact that $G_\gamma$ is an $\ell$-group and $L^kX$ is a $p$-space, we deduce that
    \[
        (\pi_0L^{n\ssm i}X)^{G_\gamma} \iso 
        (\pi_0L^n X)^{G_\gamma},
    \]
    so in particular $\gamma$ is not new. To conclude, each element  $\gamma \in L_\new^n X$ has stabilizer $G_\gamma \le G_n$ of co-dimension $>\floor{n/d}$ so its orbit must be of size divisible by $\ell^{\floor{n/d}}$.
\end{proof}

From this we immediately obtain a quantitative refinement of the $\ell$-adic continuity claim for $\lfun_X(n)$. 

\begin{prop}
    For every $d$-finite $p$-space $X$, we have
    \[
        \ell^{\floor{n/d}} \mid \cl{\lfun}_X(n) \qquad \forall n\in \NN.
    \]
    In particular, the function $\lfun_X\colon \NN \to \QQ_\ell$ admits a (unique) $\ell$-adic continuation $\hat{\lfun}_X \colon \ZZ_\ell \to \QQ_\ell$. 
\end{prop}

\begin{proof}
    By \Cref{Orbit_ell}, $\pi_0(L_\new^n X)$ is a disjoint union of $G_n$-orbits each of size divisible by $\ell^{\floor{n/d}}$. Hence, the same holds for  $\cl{\lfun}_X(n) = |\pi_0(L_\new^n X)|$.
\end{proof}

\begin{rem}
    In fact, since the $\ell$-adic valuation of $\cl{\lfun}_X(n)$ increases at least \textit{linearly} in $n$, the extension of $\lfun_X$ to $\ZZ_\ell$ is even locally \textit{analytic}. 
\end{rem}

\subsubsection{Extrapolation to $-1$}

Having shown that the function $\lfun_X$ admits an $\ell$-adic continuation $\hat{\lfun}_X\colon \ZZ_\ell \to \QQ_\ell$ for every $\pi$-finite $p$-space $X$, it remains to show that $\hat{\lfun}_X(-1) = |X|$. By Mahler's theorem combined with \Cref{fbar_new}, we have
\[
    \hat{\lfun}_X(-1) =
    \sum_{n=0}^\infty (-1)^n\cl{\lfun}_X(n) =
    \sum_{n=0}^\infty (-1)^n|\pi_0(L_\new^n X)|.
\]
However, it is not clear how to show that this series of integers converges $\ell$-adically to the rational number $|X|$. We therefore choose an indirect approach and define another function:
\[
    \Lfun_X(n) = |L^n X| \qin \QQ.
\]

The relation between $\Lfun_X$ and $\lfun_X$ is as follows:

\begin{prop}
    For every $\pi$-finite space $X$ and $n \ge 0$, we have 
    $\lfun_X(n) = \Lfun_X(n+1)$.
\end{prop}
\begin{proof}
    For every connected $\pi$-finite space $X$, the fiber sequence
    \[
        \Omega X \too LX \too X
    \]
    implies that 
    \[
        |LX| = |X|\cdot |\Omega X| = |X| \cdot |X|^{-1} = 1.
    \]
    For $X = X_1 \sqcup \dots \sqcup X_r$ a union of connected components we have
    \[
        |LX| = |LX_1 \sqcup \dots \sqcup LX_r| = 
        |LX_1| + \dots + |LX_r| = |\pi_0 X|.
    \]
    It follows that for every $n\ge 0$ we have
    \[
        \Lfun_X(n+1) = |L^{n+1}X| = |L(L^n X)| = |\pi_0(L^n X)| = \lfun_X(n). 
    \]
\end{proof}

Namely, $\Lfun_X$ has the values of $\lfun_X$ shifted one place up and $\Lfun_X(0) = |X|$. Hence, to prove that $\hat{\lfun}_X(-1) = |X|$, it suffices to prove that $\Lfun_X$ satisfies Mahler's criterion. The difficulty lies in the fact that the inverse binomial transform of $\Lfun_X$,
\[
    \cl{\Lfun}_X(n) =  \sum_{k=0}^n (-1)^{n-k} \binom{n}{k}|L^k X|,
\]
is harder to analyze than that of $\lfun_X$, and we shall use what we already proved for the latter to study the former.  
Recall that we have $LX = L_\new X \sqcup L_\old X$ and clearly $L_\new L_\old \simeq L_\old L_\new$, so we get
\[
    |L^k X| = 
    \sum_{m=0}^k \binom{k}{m} |L_\old^{k-m} L_\new^{m} X|.
\]
We remind that by $L^k_\old X$ we mean the $k$-th iterate of $L_\old$ applied to $X$ so these are the ``very old'' element in $L^k X$.
We shall hence define another auxiliary function
\(
    \ofun_X(k) = |L_\old^k X|
\)
with inverse binomial transform 
\[
    \cl{\ofun}_X(k) = 
    \sum_{m=0}^k (-1)^{k-m}\binom{k}{m}\ofun_X(m).
\]
We get
\begin{lem}\label{Fbar_Gbar}
    For every $\pi$-finite space $X$, 
    \[
        \cl{\Lfun}_X(n) = \sum_{m=0}^n \binom{n}{m}\cl{\ofun}_{(L_\new^m X)}(n-m).
    \]
\end{lem}
\begin{proof}
    Unwinding the definitions,
    \[
        \cl{\Lfun}_X(n) =  \sum_{k=0}^\infty (-1)^{n-k} \binom{n}{k}|L^k X| =
    \]
    \[
        \sum_{k=0}^\infty (-1)^{n-k} \binom{n}{k}\Bigg(\sum_{m=0}^\infty \binom{k}{m} |L_\old^{k-m} L_\new^{m} X|\Bigg) = 
    \]
    \[
        \sum_{k=0}^n\sum_{m=0}^k (-1)^{n-k} \binom{n}{k}\binom{k}{m}\ofun_{L_\new^{m} X}(k-m).
    \]
    Using the binomial identity
    \[
        \binom{n}{k}\binom{k}{m} = 
        \binom{n}{m}\binom{n-m}{k-m}
    \]
    we get
    \[
        \sum_{k=0}^\infty\sum_{m=0}^\infty (-1)^{n-k} \binom{n}{m}\binom{n-m}{k-m}\ofun_{(L_\new^{m} X)}(k-m) = 
    \]
    \[
        \sum_{m=0}^\infty \binom{n}{m} \Bigg(\sum_{k=0}^\infty (-1)^{n-k} \binom{n-m}{k-m}\ofun_{(L_\new^{m} X)}(k-m)\Bigg).
    \]
    Re-indexing so that $k-m$ becomes $k$ we finally obtain,
    \[
        \sum_{m=0}^\infty \binom{n}{m} \Bigg(\sum_{k=0}^\infty (-1)^{n-m-k} \binom{n-m}{k}\ofun_{(L_\new^{m} X)}(k)\Bigg) = 
        \sum_{m=0}^\infty \binom{n}{m}\cl{\ofun}_{(L_\new^m X)}(n-m).
    \]    
\end{proof}

It will thus suffice to bound from below the $\ell$-adic valuation of $\cl{\ofun}_{(L_\new^m X)}(n-m)$, uniformly in $m$, by a function that grows to infinity as $n\to \infty$.  We begin by analyzing $\ofun_X$.

\begin{lem}\label{G_Formula}
    For every connected $d$-finite space $X$, we have 
    \[
        \ofun_X(n) =
        \prod_{m=1}^d |\pi_m(X)|^{-\binom{n-1}{m-1}}.
    \]
\end{lem}
\begin{proof}
    Observe that we have natural isomorphisms 
    \[
        \Omega LY \simeq 
        L\Omega Y \simeq
        \Omega Y \times \Omega^2 Y
    \]
    for every pointed space $Y$.
    Using the fact that $L_\old^n X$ is connected and applying the above inductively, we compute
    \[
        \ofun_X(n) = 
        |L^n_\old X| = 
        |\Omega L^n_\old X|^{-1} = 
        |\Omega L^n X|^{-1} =
        |L^n \Omega X|^{-1} =
        \prod_{k=0}^n |\Omega^{k+1} X|^{-\binom{n}{k}} =
    \]
    \[
        \prod_{k=0}^\infty \Big(\prod_{m = 0}^\infty |\pi_{m}(\Omega^{k+1} X)|^{(-1)^m}\Big)^{-\binom{n}{k}} =
        \prod_{k=0}^\infty \Big(\prod_{m = 0}^\infty |\pi_{m+k+1}(X)|^{(-1)^m}\Big)^{-\binom{n}{k}} =
    \]
    \[
        \prod_{k=0}^\infty \Big(\prod_{m = k+1}^\infty |\pi_m(X)|^{(-1)^{m+k+1}}\Big)^{-\binom{n}{k}} =
        \prod_{m=1}^\infty \Big(\prod_{k = 0}^{m-1} |\pi_m(X)|^{(-1)^{m+k+1}}\Big)^{-\binom{n}{k}} =
    \]
    \[
        \prod_{m=1}^\infty \Big(\prod_{k = 0}^{m-1} |\pi_m(X)|^{(-1)^{m+k+1}}\Big)^{-\binom{n}{k}}=
        \prod_{m=1}^\infty |\pi_m(X)|^{(-1)^m\Big(\sum_{k=0}^{m-1}\binom{n}{k}(-1)^k\Big)} = 
    \]
    \[
        \prod_{m=1}^\infty |\pi_m(X)|^{-\binom{n-1}{m-1}} = 
        \prod_{m=1}^d |\pi_m(X)|^{-\binom{n-1}{m-1}}.
    \]
    In the second to last euality we have used the binomial identity
    \[
        \sum_{k=0}^{m-1}\binom{n}{k}(-1)^k = 
        (-1)^{m-1} \binom{n-1}{m-1}.
    \]
\end{proof}

Assuming further that $X$ is a $p$-space, we get the following consequence:

\begin{cor}\label{G_Expoly}
    For every connected $d$-finite $p$-space $X$, we have 
    \[
        \ofun_X(n) = p^{f(n)},
    \]
    for some (integer valued) polynomial $f(x) \in \QQ[x]$ of degree $d$.
\end{cor}

\begin{proof}
    By \Cref{G_Formula}, we can take 
    \[
        f(x) = -\sum_{m=0}^d c_m \binom{x-1}{m-1}
    \]
    for $c_m = \log_p|\pi_m(X)|$.
\end{proof}

To analyze the inverse binomial transform of such functions we shall use a simple observation about the inverse binomial transform of polynomials.

\begin{lem}\label{Polynomial_Inv_Binomial}
    For every polynomial $f(x) \in \QQ[x]$ of degree $d$, the inverse binomial transform of $f(n)$ satisfies $\cl{f}(n) = 0$ for $n\ge d+1$.
\end{lem}
\begin{proof}
    The inverse binomial transform can be computed by $\cl{f}(n) = (\Delta^n f)(0)$, where $\Delta$ is the forward difference operator
    \[
        (\Delta f)(n) = f(n+1) - f(n).
    \]
    For a degree $d$ polynomial $f$, the forward difference $\Delta f$ is a polynomial of degree $d-1$, hence the claim follows by induction.
\end{proof}

\begin{prop}\label{G_ell}
    For every $d$-finite $p$-space $X$, we have
    \[
        \ell^{\floor{n/d}} \mid \cl{\ofun}_X(n) \qquad \forall n\in \NN.
    \]
\end{prop}
\begin{proof}
    By \Cref{G_Expoly}, we have $\ofun_X(n) = p^{f(n)}$ for some degree $d$ polynomial. Write $p = 1 + \ell N$ and expand
    \[
        \ofun_X(n) = (1+ \ell N)^{f(n)} = 
        1 + \binom{f(n)}{1}\ell N + \binom{f(n)}{2}\ell^2 N^2 + \binom{f(n)}{3}\ell^3 N^3 + \dots 
    \]
    Since $\binom{f(x)}{k}$ is a polynomial of degree $dk$, by \Cref{Polynomial_Inv_Binomial}, its inverse binomial transform vanishes for $n\ge dk + 1$. By the linearity of the binomial transform we get that for all such $n$, the function $\cl{G}_X(n)$ agrees with the inverse binomial transform $\cl{G}_{X,k+1}(n)$ of the tail of the above expansion
    \[
        \ofun_{X,k+1}(n) := \binom{f(n)}{k+1}\ell^{k+1}N^{k+1} +
        \binom{f(n)}{k+2}\ell^{k+2}N^{k+2} +
        \dots,
    \]
    which is divisible by $\ell^{k+1}$, implying the claim.
\end{proof}

We are now in good shape to prove the $\ell$-adic continuity of the function $\lfun_X(n) := |L^n X|$ by combining our estimate for $\cl{\ofun}_X(n-m)$ with the estimate for $\cl{\lfun}_X(m)$ from the previous subsection.

\begin{thm}\label{Extrapolation_Minus_One}
    For every $d$-finite $p$-space $X$, the inverse binomial transform of the function $\Lfun_X(n) = |L^n X|$ satisfies
    \[
        \ell^{\floor{n/d}} \mid \cl{\Lfun}_X(n) \qquad \forall n\in \NN.
    \]
    Consequently, $\Lfun_X$ is $\ell$-adically continuous and hence $\hat{\lfun}_X(-1) = \Lfun_X(0) = |X|$.
\end{thm}

\begin{proof}
    By \Cref{Fbar_Gbar}, we have
    \[
        \cl{\Lfun}_X(n) = \sum_{m=0}^n \binom{n}{m}\cl{\ofun}_{(L_\new^m X)}(n-m).
    \]
    it thus suffices to show that for each $m=0,\dots, n$ we have
    \[
        \ell^{\floor{n/d}} \mid \cl{\ofun}_{(L_\new^m X)}(n-m).
    \]
    Observe that
    \[
        \cl{\ofun}_{X\sqcup Y}(k) = 
        \cl{\ofun}_{X}(k) + \cl{\ofun}_{Y}(k).
    \]
    hence, 
    \[
        \cl{\ofun}_{(L_\new^m X)}(n-m) = 
        \sum_{\gamma \in \pi_0(L_\new^{m}X)}\cl{\ofun}_{(L_\gamma^m X)}(n-m),
    \]
    where $L_\gamma^m X \sseq L^m X$ is the connected component of $\gamma$. Each  $L_\gamma^m X$ is a connected $d$-finite $p$-space, so by \Cref{G_ell} we have 
    \[
        \ell^{\floor{(n-m)/d}} \mid \cl{\ofun}_{(L_\gamma^m X)}(n-m).
    \]
    On the other hand, the group $G_n \simeq (\ZZ/\ell)^n$ acts on $\pi_0(L_\new^m X)$ and if $\gamma$ and $\gamma'$ are in the same orbit, then 
    \[
        \cl{\ofun}_{(L_\gamma^m X)}(n-m) = 
        \cl{\ofun}_{(L_\gamma'^m X)}(n-m).
    \]
    We have also shown that the size of each orbit is divisible by $\ell^{\floor{m/d}}$. Thus, by \Cref{Orbit_ell}, $\cl{\ofun}_{(L_\new^m X)}(n-m)$ is divisible by 
    \[
        \ell^{\floor{m/d}}\cdot \ell^{\floor{(n-m)/d}} = 
        \ell^{\floor{m/d}+\floor{(n-m)/d}}
    \]
    and 
    \[
        \floor{m/d}+\floor{(n-m)/d} \ge \floor{n/d}.
    \]
\end{proof}

\section{Examples of $\zeta_X$}\label{Sec:Examples}

Our method for establishing the $\ell$-adic continuity of the function $\lfun_X$ and evaluating its extrapolation to $n= -1$ circumvented the need for an explicit closed-form formula for the numbers $\lfun_X(n)$ in terms of the homotopy invariants of $X$. In this section, we provide such formulae in some sufficiently simple cases, which also hint at the difficulty of establishing such formulae in complete generality.

\subsubsection{Loop spaces}

As already note in the introduction, for the fundamental example of the Eilenberg-MacLane space $ B^dC_p$ we have
\[
    \lfun_{B^dC_p}(n) = 
    |H^d(\TT^n; \FF_p)| = 
    p^{\binom{n}{d}}.
\]
Since $p \equiv 1$ modulo $\ell$, this function is $\ell$-adically continuous and its extrapolation to $n=-1$ is indeed given, independently of $\ell$, by 
\[
    p^{\binom{-1}{d}} = p^{(-1)^d} = |B^dC_p| \qin \QQ \sseq \QQ_\ell.
\] 
More generally, the situation is quite simple for arbitrary loop spaces.

\begin{example}
    For a $d$-finite loop space $X$, we have $LX \simeq X \times \Omega X$ and therefore
    \[
        \lfun_X(n) = \lfun_X(n-1)\lfun_{\Omega X}(n-1).
    \]
    Since $\Omega X$ is a $(d-1)$-finite loop space, we can solve inductively the above recursive formula,
    \[
        \lfun_X(n) = \prod_{k=0}^d |\pi_k X|^{\binom{n}{k}},
    \]
    using nothing but the binomial identity
    \[
        \binom{n}{k} = \binom{n-1}{k} + \binom{n-1}{k-1}.
    \]
    In other words, when $X$ is a loop space, $\lfun_X$ does not see the Postnikov invariants of $X$, in the sense that it does not distinguish it from the product of Eilenberg-MacLane spaces\footnote{By \cite{hopkins1998morava}, for \textit{double} loop $\pi$-finite spaces, the $K(n)$-homology itself can not see the Postnikov invariants.} 
    $\prod_{k=0}^d B^k(\pi_k X)$.
\end{example}

In contrast, for arbitrary $\pi$-finite $p$-spaces $X$, which are not loop spaces, the situation is considerably more complicated. This is dew to the fact that the connected components of $LX$ are no longer isomorphic to each other (which gets worse as the functor $L$ is iterated). As a consequence, the function $\lfun_X$ no longer depends only on the sizes of the homotopy groups of $X$. We present two families of examples in which one can give a closed form formula for $\lfun_X$, and which demonstrate the dependence of $\lfun_X$ both on the homotopy group structure and the Postnikov invariants of $X$.

\subsubsection{Classifying spaces}

The Morava-Euler characteristics of classifying spaces $BG$ of finite (not necessarily $p$-)groups $G$, were studied already by Hopkins-Kuhn-Ravenel in their celebratd work on generalized character theory \cite{hopkins2000generalized}. For a fixed (implicit) prime $p$, let $G^{(n)} \sseq G^n$ be the set of $n$-tuples of commuting elements of $p$-power order in $G$. The group $G$ acts on $G^{(n)}$ by (simultaneous) conjugation. In the first part of \cite[Theorem~B]{hopkins2000generalized}, it is shown that
\[
    \chi_n(BG) = |G^{(n)}/\conj| = |[B\ZZ_p^n, BG]|.
\]
Moreover, using M\"obius inversion on the lattice of abelian subgroups of $G$, the second part of \cite[Theorem~B]{hopkins2000generalized} provides the formula
\[
    \chi_n(BG) = \frac{1}{|G|}\sum_{A}c_A[A:A_p]|A_p|^{n+1}, 
\]
where the sum is over all abelian subgroups $A\le G$, $A_p \le A$ is the $p$-primary part (i.e. $p$-Sylow subgroup), and $c_A$ are integer coefficients determined by the recursive relation
\[
    \sum_{B \le A} c_B = 1.
\]
In particular, it is clear that $\chi_n(BG)$ is $\ell$-adically continuous and its extrapolation to $n = -1$ is given by the following:

\begin{defn}
    For every finite group $G$ and a prime $p$, we set the \textit{$p$-typical homotopy cardinality} of $BG$ to be
    \[
        |BG|_p := \frac{1}{|G|}\sum_{A}c_A[A:A_p].
    \]
\end{defn}

In particular, when $G$ is a $p$-group, the above formula for $\chi_n(BG)$ is of the form:
\[
    \lfun_{BG} := \chi_n(BG) = \frac{1}{|G|}\big(c_1 p^{a_1(n+1)} + \dots +c_r p^{a_r(n+1)}\big),
\]
for some  $a_1, \dots, a_r \in \NN$ and $c_1,\dots,c_r \in \ZZ$. It is also evident that $\sum c_i = 1$, giving
\[
    |BG|_p =  \hat{\lfun}_{BG}(-1) = \frac{1}{|G|} = |BG|.
\] 
in accordance with \Cref{Intro_Morava_Euler_Cont}.

\begin{example}
    There are $3$ abelian groups $G$ of order $8$:
    \[
        \ZZ/8 \quad,\quad Z/4 \times \ZZ/2 \quad,\quad (\ZZ/2)^3.
    \]
    As we already know, for all of them we have
    \[
        \lfun_{BG(n)} = \frac{1}{8}\cdot 8^{n+1}.
    \]
    The remaining (non-commutative) groups of order $8$ are the dihedral group $D_4$ and the quaternion group $Q_8$. For both, all the proper subgroups are abelian. In each case, there are 3 maximal ones of order $4$, and the intersection of each two is the center, which is of order 2. Thus, we get 
    \[
        \lfun_{BD_4}(n) = 
        \lfun_{BQ_8}(n) =
        \frac{1}{8}(3\cdot 4^{n+1} - 2 \cdot 2^{n+1}).
    \]
\end{example}

When $G$ is not a $p$-group the situation is a bit different. First, if $G$ is nilpotent, it is a product of $q$-groups $G_q$ for various primes $q$, and we get $f_{BG} = f_{BG_p}$. In particular,
\[
    |BG|_p = |BG_p| = \frac{1}{|G_p|}.
\]
So, $|BG|_p$ is in a precise sense the $p$-part of $|G|$. 
\begin{rem}
    More generally, by \cite{prasma2017sylow}, every nilpotent $\pi$-finite space $X$ splits uniquely as a (finite) product 
    \[
        X \simeq \prod_{q\text{ prime}} X_q \qin \Spc,
    \] 
    with $X_q$ a $\pi$-finite $q$-space. We thus have $\lfun_X = \lfun_{X_p}$ (and hence $|X|_p = |X_p|$).
\end{rem}
For a general finite group $G$, the number $|BG|_p$ need not be even a power of $p$, but we still have the following:

\begin{prop}\label{Prob_Prime2p}
    For $G$ a finite group, $|BG|_p$ is the probability of a random element of $G$ to be of order prime to $p$. Moreover, $|BG|_p$ has the same $p$-valuation as $|BG|$.
\end{prop}
\begin{proof}
    Every element $g \in G$ lies in some abelian subgroup $A \le G$. The number of elements of order prime to $p$ in every such $A$ is $[A\colon A_p]$. Thus, by M\"obius inversion, 
    \(
        \sum_{A}c_A[A:A_p]
    \)
    is exactly the number of elements of order prime to $p$ in $G$. For the second part, we need to show that the collection $G^{(1)} \sseq G$ of elements of order prime to $p$ is itself of size prime to $p$. Let $P \triangleleft G$ be a $p$-Sylow acting by conjugation on $G^{(1)}$. By the orbit-stabilizer theorem, we have 
    \[
        |\mathrm{Fix}_P(G^{(1)})| \equiv |G^{(1)}| \mod p.
    \]
    Setting $H = C_G(P)$, we get $\mathrm{Fix}_P(G^{(1)}) = H^{(1)}$. Thus, if $H$ is smaller then $G$, we can proceed by induction. Otherwise, we get that $P$ is central in $G$ and hence $G \simeq P \times K$ for some group $K$ of order prime to $p$, in which case the claim is obvious.
\end{proof}
 
We conclude by giving explicit formulae for $\chi_n(BG)$ and $|BG|_p$ in the case of the symmetric groups $G = \Sigma_m$. Thhe former are most neatly expressed in the form of a generating function using the so called \textit{Gaussian (a.k.a.  $q$-)binomial coefficients}:
\[
    \qbinom{n}{k}_q = 
    \frac{(1-q^n)(1-q^{n-1})\dots(1-q^{n-k+1})}{(1-q^k)(1-q^{k-1})\dots(1-q)}
\]
The following, and more elaborate versions thereof, can be found in \cite{tamanoi2001generalized}. For completeness, we provide the (easy) direct argument.
\begin{prop}\label{GenFun_Chi_BSm}
    For all $n \ge 0$, we have
    \[
        \sum_{m=0}^\infty \chi_{n}(B\Sigma_m) x^m = 
        \prod_{r = 0}^\infty (1-x^{p^r})^{-\qbinom{r+n-1}{r}_p}
    \]
\end{prop}
\begin{proof}
    By Burnside's lemma, we have
    \[
        |G^{(n)}/\conj| = 
        \frac{1}{|G|}\sum_{\vec{g} \in G^{(n)}} |C_G(\Vec{g})|.
    \]
    Observe that $|C_G(\Vec{g})|$ is the number of ways to to extend the commuting $n$-tuple $\vec{g} = (g_1,\dots,g_n)$ of $p$-power elements to a commuting $(n+1)$-tuple $(g_1,\dots,g_n,h)$, where $h \in G$ is arbitrary. In other words, we get
    \[
        \chi_n(BG) = \frac{1}{|G|}|\hom(\ZZ_p^n \times \ZZ, G)|.
    \]
    By \cite[Exercise 5.13]{gessel2002enumerative}, for every finitely generated group $H$, we have
    \[
        \sum_{m=0}^\infty \frac{|\hom(H \times \ZZ, \Sigma_m)|}{m!}x^m = 
        \prod_{d=1}^\infty (1-x^d)^{-c_d(H)},
    \]
    where $c_d(H)$ is the number of conjugacy classes of index $d$ subgroups of $H$. Thus, putting $H = \ZZ_p^n$, we get\footnote{While $\ZZ_p^n$ is not quite finitely generated, replacing it with $(\ZZ/p^N)^n$ for $N\gg 0$ does not change the first (arbitrarily many) first terms of both sides of the equation, so we can still apply the formula to it.}
    \[
        \sum_{m=0}^\infty \chi_n(B\Sigma_m) x^m = 
        \sum_{m=0}^\infty \frac{|\hom(\ZZ_p^n \times \ZZ, \Sigma_m)|}{m!} =
        \prod_{r = 0}^\infty (1-x^{p^r})^{-c_{p^r}(\ZZ_p^n)}.
    \]
    Finally, we have (see, e.g., \cite{gruber1997alternative})
    \[
        c_{p^r}(\ZZ_p^n) = 
        \qbinom{r+n-1}{r}_p.
    \]    
\end{proof}

\begin{example}
    For $n=1$, we get the generating function
    \(
        \prod_{r\ge 0}\frac{1}{1-x^{p^r}}.
    \)
    This is the $p$-typical analogue of the usual \textit{partition function}, and it counts the number of partitions of $m$ with parts of $p$-power order. This can be thought of as the number possible cycle types of a permutation on $m$ elements, such that all the cycles have length prime to $p$, which indeed coincides with $|\Sigma_m^{(1)}/\conj|$.
\end{example}

We can also give a very simple formula for the $p$-typical homotopy cardinality of $B\Sigma_m$.

\begin{prop}\label{pTypical_Card_Sm}
    For every natural number $m$ and a prime $p$, we have
    \[
        |B\Sigma_m|_p = 
        \prod_{
        \substack{
             1 \le d \le m,  \\
            p\mid d 
        }}
        \Big( 1-\frac{1}{d} \Big).
    \]
\end{prop}
\begin{proof}
    By  \Cref{Prob_Prime2p}, $ |B\Sigma_m|_p$ is the probability of a random permutation on $m$ elements to be of order prime to $p$, or equivalently, with all cycles of length prime to $p$. This is well known to be given by the above formula (see \cite[Theorem 2]{bolker1980counting}).
\end{proof}

For an odd $p$, we can use $\ell$-adic extrapolation to deduce the above formula for $|B\Sigma_m|_p$ from \Cref{GenFun_Chi_BSm}. Indeed, it is not hard to see that the $p$-binomial coefficient $\qbinom{a}{b}_p$ are $\ell$-adically continuous in $a$. In fact, so is the whole expression
\[
     \prod_{r = 0}^\infty (1-x^{p^r})^{-\qbinom{r+n-1}{r}_p}
\]
as a function of $n$. To evaluate it at $n=-1$ we observe that the only non-trivial terms are for $r=0,1$, in which case we have (see \cite{formichella2019gaussian} for more on $q$-binomial coefficients with negative arguments),
\[
    \qbinom{-2}{0}_p = 1 \quad,\quad 
    \qbinom{-1}{1}_p = -\frac{1}{p}.
\]
We thus get,
 \[
    \sum_{m=0}^{\infty} |B\Sigma_m|_p\: x^m = 
    \frac{(1-x^p)^{1/p}}{1-x},
\]
from which \Cref{pTypical_Card_Sm} readily follows by a standard Taylor series expension.
 

\subsubsection{Cup-power extensions}

As mentioned above, for a general $\pi$-finite $p$-space $X$ the function $\lfun_X$ would depend on the Postnikov invariants of $X$, as well as its homotopy groups. We shall demonstrate this by computing $\lfun_X$ for a certain family of spaces with non-trivial Postnikov invariants. The spaces $B^d C_p$ represent $\FF_p$-cohomology and so the operation
\[
    H^d( - ; \FF_p) \oto{x \mapsto x^{\cup m}} H^{md}( - ; \FF_p)
\]
is represented by a certain map $P^m \colon B^dC_p \to B^{2d}C_p$. We let $X_{d,m}$ be the homotopy fiber of this map, so that we have a fiber sequence
\[
    X_{d,m} \too B^dC_p \oto{\ P^m\ } B^{md}C_p.
\]
From the long exact sequence in homotopy groups we immediately deduce that $X_{d,m}$ is connected and that 
\[
    \pi_i(X_{d,m}) \simeq 
    \begin{cases}
        C_p & i=d,md-1 \\
        0 & \text{else}.
    \end{cases}
\]

We begin by reducing the problem to combinatorics. Let $s_{d,m}(n)$ be the number of elements 
\[
    \omega \in \bigwedge^d \FF_p\{x_1,\dots,x_n\},
    \quad \text{such that} \quad \omega^{\wedge m} = 0.
\] 
We have the following:
\begin{prop}
    For all $n \ge 0$ we have
    \[
        |L^n X_{d,m}| = 
        p^{\binom{n-1}{md-1}-\binom{n-1}{d-1}} s_{d,m}(n).
    \]
\end{prop}
\begin{proof}
    For every $n \in \NN$ we get a fiber sequence
    \[
        L^n X_{d,m} \too L^n B^dC_p \oto{\ P^m \ } L^n B^{md} C_p.
    \]
    Restricting to the connected component of the base point of $L^n B^{md} C_p$ and its preimage under $P^m$, we get a fiber sequence
    \[
        L^n X_{d,m} \too L_0^d BC_p \oto{\ P^m \ } L_\old^n B^{md} C_p
    \]
    with a connected base. We observe that on $\pi_0$ the map $P^m$ takes an element 
    \[
        \omega \in \bigwedge^d \FF_p\{x_1,\dots,x_n\} =
        H^d(\TT^n ; \FF_p) = \pi_0(L^n B^d C_p)
    \]
    to its $m$-th power. Thus, $|\pi_0(L_0^nB^d C_p)| = s_{d,m}(n)$, so we get
    \[
        |L^n X_{d,m}| = 
        |L^n_\old B^{md}C_p|^{-1}|L_\old^n B^dC_p| s_{d,m}(n) = 
        p^{\binom{n-1}{md-1}-\binom{n-1}{d-1}} s_{d,m}(n).
    \]
\end{proof}

For $d=2$ and $p\neq 2$ we can work out explicitly the value of $s_{d,m}(n)$ and hence of $f_{X_{d,m}}(n)$. 

\begin{prop}
    For $p\neq 2$ and all $m\ge1$, we have
    \[
        \lfun_{X_{2,m}}(n) = 
        p^{\binom{n}{2m-1}-n}
        \sum_{k=0}^{m-1} 
        p^{k(k-1)} \bigg( \prod_{i=1}^k(p^{2i - 1} - 1) \bigg) \qbinom{n+1}{2k}_p.
    \]
\end{prop}
\begin{proof}
    To every $2$-form $\omega$ we have an associated skew-symmetric matrix $A \in M_n(\FF_p)$, and $\omega^m = 0$, if and only if $\rk(A) < 2m$.  The number of skew-symmetric $n\times n$ matrices of rank $2k$ is given by (e.g. \cite[Proposition 3.8]{lewis2010matrices})
    \[
        r_k(n) = p^{k(k-1)} \bigg( \prod_{i=1}^k(p^{2i - 1} - 1) \bigg) \qbinom{n}{2k}_p,
    \]
    and 
    \[
        s_{2,m}(n) = r_0(n) + r_2(n) + \dots + r_{2m-2}(n).
    \]
    We thus get
    \[
        |L^n X_{2,m}| = 
        p^{\binom{n-1}{2m-1}-n+1}\sum_{k=0}^{m-1}r_{2k}(n) = 
        p^{\binom{n-1}{2m-1}-n+1}
        \sum_{k=0}^{m-1} 
        p^{k(k-1)} \bigg( \prod_{i=1}^k(p^{2i - 1} - 1) \bigg) \qbinom{n}{2k}_p.
    \]
    Finally, the claim follows from the relation $\lfun_X(n) = |L^{n+1}X|$.
\end{proof}

\begin{example}    
    For $m =2$, we have $\pi_2(X_{2,2}) = \pi_3(X_{2,2}) = C_p$ with all other homotopy groups zero, so in particular $|X_{2,2}| = 1$. On the other hand, 
    \[
        \lfun_{X_{2,2}}(n) = \frac{p^{n+1} + p^{2-n} - p - 1}{p^2 - 1}
        p^{\binom{n}{3}},
    \]
    which indeed satisfies $\hat{\lfun}_{X_{2,2}}(-1) = 1$.
\end{example}

For higher values of $d$, the problem of determining explicitly $s_{d,m}(n)$, and hence $\lfun_{X_{d,m}}(n)$, seems to be considerably more difficult.

\begin{rem}\label{Expolynomiality}
    In all the explicit examples we had so far, the function $\lfun_X(n)$ had the form of a linear combination over the rational numbers of functions of the form $p^{f(n)}$, where $f(n)$ is an (integer valued) polynomial over the rational numbers. We leave it to the reader to decide whether this is an indication of a general phenomena, or the limitation of our computational ability. In particular, it would be interesting to know whether this is the case for $\lfun_{X_{d,m}}(n)$ (equivalently $s_{d,m}(n)$) when $d > 2$.
\end{rem}

\section{$\ell$-Adically Continuous Resolutions}\label{Mysteries}

So far, we have only discussed the extent to which the classical Euler characteristic and the homotopy cardinality function of Baez and Dolan are compatible (in terms of the existence of a common extension). However, our results can also offer some insights into the heuristics for why these numbers ``want to be the same'' even when formally only one of them is well-defined. As mentioned in the introduction, the most basic such heuristic concerns the classifying space of a finite group $G$, where the standard simplicial structure on $BG$ yields the divergent series expression
\[
    \chi(BG) = \sum_{n = 0}^\infty (-1)^n(|G|-1)^n. 
\]

One usually proceeds to evaluate this divergent sum to the desired result $1/|G|$ using some sort of a regularization technique, such as the analytic continuation to the point $t = 1$ of the power series function
\[
    f(t) := 
    \sum_{n = 0}^\infty (-1)^n(|G| - 1)^nt^n = 
    \frac{1}{1 + (|G| - 1)t}.
\]
We observe that when $G$ is a $p$-group and $\ell \mid (p-1)$, the series defining $\chi(BG)$ also converges \textit{$\ell$-adically} to $1/|BG|$, without any need for regularization. Furthermore, it can be viewed as a continuity result for the generalized homotopy cardinality function. Namely, both $BG$, being a $\pi$-finite $p$-space, and its simplicial skeleta $\sk^n(BG)$, being finite spaces, are in $\Spc^{\psmall}$, and we have
\[
    |\sk^n(BG)| \oto{n\to \infty} 
    |BG| \qin \QQ \sseq \QQ_\ell.
\]
This formulation is not only more inline with the perspective taken in this note, but is in fact a special case of a more general continuity phenomenon. 

\begin{defn}
    Given $X\in \Spc^{\psmall}$, an \textit{$\ell$-adically continuous resolution} of $X$ is a simplicial diagram $X_\bullet \colon \Delta^\op \to \Spc^{\psmall}$, such that $\colim X_\bullet = X$, and the sequence $x_n = |X_n|$ extends to a continuous function $\ZZ_\ell \to \QQ_\ell$ whose value at $-1 \in \ZZ_\ell$ is $|X|$.
\end{defn}

If we make $X_\bullet$ an augmented simplicial space by appending $X_{-1} := X$, the last condition reads suggestively as
\[
    \lim_{n \to -1}x_n = x_{-1}.
\]
We shall show that for an $\ell$-adically continuous resolution of $X$, the partial realizations are also in $\Spc^{\psmall}$, and their cardinalities converge $\ell$-adically to $|X|$. We supplement this by establishing  $\ell$-adical continuity for a general family of simplicial resolutions including the bar construction of a finite $p$-group. 

We begin with a general observation regarding the cardinalities of partial realizations of simplicial objects.

\begin{prop}\label{Cnt_Simp_Binomial}
    Let $X_\bullet \colon \Delta^\op \to \Spc^{\psmall}$. The $n$-th partial realizations $\sk^n X$ are in $\Spc^{\psmall}$ and their cardinalities are given by 
    \[
        |\sk^n X| = \sum_{k=0}^n (-1)^k \cl{x}_k \qin \QQ,
    \]
    where $(\cl{x}_k)$ is the inverse binomial transform of the sequence $x_n = |X_n|$.
\end{prop}
\begin{proof}
    By definition, $\sk^n X$ is the colimit of the restriction of $X_\bullet$ along $\Delta_{\le n}^\op \into \Delta^\op$.
    While the diagram shape $\Delta_{\le n}$ is not itself finite, it admits a right cofinal map from the full subcategory $\mathcal{J}_n \sseq (\Delta_{\le n})_{/[n]}$ spanned by the injective maps $[m] \into [n]$ (see \cite[Lemma 1.2.4.17]{HA}). We can thus compute the partial realizations of $X_\bullet$ using these finite diagrams as follows
    \[
        \sk^n X = \colim_{\mathcal{J}_n^\op} X_\bullet.
    \]
    As an illustration, the cases $n=0,1,2$ amount to the following co-cartesian 1-, 2-, and 3-dimensional cubes:
    
    \[\begin{tikzcd}
    	{X_0} && {X_1} && {X_0} && {X_2} && {X_1} \\
    	& {,} &&&& {,} && {X_1} && {X_0} \\
    	{\sk^0 X} && {X_0} && {\sk^1 X} && {X_1} && {X_0} \\
    	&&&&&&& {X_0} && {\sk^2 X.}
    	\arrow[from=1-7, to=3-7]
    	\arrow[from=1-7, to=1-9]
    	\arrow["{\text{ }}"{description}, from=3-7, to=3-9]
    	\arrow[from=1-9, to=3-9]
    	\arrow[from=1-9, to=2-10]
    	\arrow[from=2-10, to=4-10]
    	\arrow[from=3-9, to=4-10]
    	\arrow[from=3-7, to=4-8]
    	\arrow[from=4-8, to=4-10]
    	\arrow[from=1-7, to=2-8]
    	\arrow[from=2-8, to=4-8]
    	\arrow[from=2-8, to=2-10]
    	\arrow[from=1-3, to=3-3]
    	\arrow[from=3-3, to=3-5]
    	\arrow[from=1-5, to=3-5]
    	\arrow[from=1-3, to=1-5]
    	\arrow["\wr", from=1-1, to=3-1]
    	\arrow["\lrcorner"{anchor=center, pos=0.125, rotate=180}, draw=none, from=3-5, to=1-3]
    \end{tikzcd}\]

    Applying iteratively the additivity  (i.e. inclusion-exclusion) of the homotopy cardinality, we get the formula
    \[
        |\sk^n X| =  \sum_{k=0}^{n} (-1)^k \binom{n}{k+1}x_k,
    \]
    and hence
    \[
        (-1)^n(|\sk^n X| - |\sk^{n-1} X|) = 
        \sum_{k=0}^{n} (-1)^{n-k} \binom{n}{k+1}x_k - 
        \sum_{k=0}^{n-1} (-1)^{n-k} \binom{n-1}{k+1}x_k = 
    \]
    \[
        \sum_{k=0}^{n}(-1)^{n-k} 
        \bigg(\binom{n}{k+1} - \binom{n-1}{k+1}\bigg)x_k = 
        \sum_{k=0}^{n}(-1)^{n-k} \binom{n}{k}x_k =
        \cl{x}_n.
    \]
    The claim follows by rearranging the above expression.
\end{proof}

We obtain the following consequence for $\ell$-adically continuous resolutions:

\begin{cor}
    For an $X_\bullet \colon \Delta^\op \to \Spc^{\psmall}$ be an $\ell$-adically continuous resolution of $X$ we have,
    \[
        |\sk^n X| \oto{n\to \infty} |X| \qin \QQ_\ell.
    \]
\end{cor}

Thus, the number $|X|$, which is by definition the $\ell$-adic limit of $|X_n|$ as $n \to -1$, is also the $\ell$-adic limit of $|\sk^n X|$ as $n\to \infty$, 

\begin{proof}
    By \Cref{Cnt_Simp_Binomial} and Mahler's theorem we have
    \[
        \lim_{n \to \infty} |\sk^n X| = 
        \sum_{n=0}^\infty (-1)^n \cl{x}_n =
        \sum_{n=0}^\infty \binom{-1}{n} \cl{x}_n = 
        \lim_{n \to -1}x_n = |X|.
    \]    
\end{proof}

One particular source of simplicial resolutions is the Check resolution $\check{C}(f)_\bullet$ of a map $f \colon X\to Y$ given object-wise by the iterated fiber-product
\[
    \check{C}(f)_n = \overset{(n+1) \text{ times}}{\overbrace{X \times_Y X \times_Y \dots \times_Y X}}.
\]
With all the preliminary work done above, we easily get:

\begin{prop}\label{Check_Cnt}
    Given a fiber sequence of  $\pi$-finite $p$-spaces,
    \(
        X \oto{\ f\ } Y \too Z
    \)
    with $Z$ connected, $\check{C}(f)_\bullet$ is an $\ell$-adically continuous resolution of $Y$. In particular, 
   \[
        \lim_{n \to \infty} |\sk^n \check{C}(f)| \too |Y|.
   \]
\end{prop}

\begin{proof}
    Since $Z$ is connected, $f$ is surjective on $\pi_0$, so $\colim \check{C}(f)_\bullet \simeq Y$. On the other hand, 
    let $G = \Omega Z$ be the corresponding $\pi$-finite $p$-group, so that $Y = X/\!\!/G$. The Check resolution of $f$ assumes the form
    \(
        \check{C}(f)_n \simeq X \times G^n,
    \)
    and since all homotopy groups of $G$ are $p$-groups, the sequence 
    \[
        x_n := |\check{C}(f)_n| = |X| |G|^n
    \]
    extends to an $\ell$-adically continuous function on $\ZZ_p$, whose value at $n=-1$ is 
    \[
        |X||G|^{-1} = |X||Z| = |Y|. 
    \]
\end{proof}

Going back to our motivating example, for a finite $p$-group $G$, the Check resolution of the base-point map $\pt \to BG$ produces the bar construction on $G$, which is a simplicial set $X_\bullet$ with $|G|^n$ simplices in degree $n$, and whose colimit is $BG$. This reproduces the fact that the series $\sum_{n=0}^\infty (-1)^n (|G| - 1)^n$, where $(|G| - 1)^n$ is the number of \textit{non-degenerate} $n$-simplices, converges $\ell$-adically to $1/|G|$, which is also the limit $n \to -1$ of the function $|G|^n$. However, the theory is applicable in greater generality. 

\begin{example}
    If $G$ is an \textit{abelian} $p$-group, $BG$ is again a $\pi$-finite $p$-group, and we can use the bar construction iteratively to produce a $d$-fold simplicial set with $|G|^{n_1n_2\dots n_d}$ simplices in the multi-degree $(n_1,n_2,\dots,n_d)$, whose colimit is $B^d G$.  Passing to the diagonal, we get a simplicial set $X_\bullet$ with $|G|^{n^d}$ simplices in degree $n$, and whose colimit is still $B^d G$. We have,
    \[
        |G|^{n^d} \oto{n\to -1} |G|^{(-1)^d} = |B^d G|.
    \]
    Hence, $X_\bullet$ is an $\ell$-adically continuous resolution of $B^d G$.  Consequently, we get $|\sk^n (B^d G)| \to |B^d G|$, which can be thought of as computing the homotopy cardinality of the space $B^dG$ by the infinite alternating sum of number of non-degenerate simplices in each dimension of the above simplicial structure.
\end{example}

\begin{rem}
    In the previous example, the number of non-degenerate simplices in the simplicial set $X_\bullet$ realizing to $B^dG$ is of order of magnitude $p^{n^d}$. The super-exponential growth of these numbers prevents them from being summable by ordinary (archimedian) regularization techniques, such as analytic continuation of power series. 
\end{rem}

We conclude with a discussion of a more elaborate family of examples of $\ell$-adically continuous resolutions coming form bar constructions of suitably finite \textit{simplicial groups}. Recall that the associated Moore complex of a simplicial group $G_\bullet$, defined by
\[
    N_n G = \bigcap_{i=1}^n \ker(G_n \oto{d_i} G_{n-1}),
\]
with the remaining $d_0$ differential, is a (non-abelian) chain-complex of groups with ``homology'' given by $\pi_n(G)$. 

\begin{defn}
    A \textit{$\pi$-finite simplicial $p$-group} is a simplicial group $G_\bullet$, such that $N_nG$ is a finite $p$-group for all $n$, which vanishes for $n \gg 0$. 
\end{defn}

The realization of a $\pi$-finite simplicial $p$-group $G_\bullet$ is a $\pi$-finite $p$-group $G$. Moreover, we get a degree-wise finite simplicial model for the $\pi$-finite $p$-space $BG$, using the bar construction $(BG)_\bullet$ given by the diagonal of the bi-simplicial set with $(BG_n)_m = (G_n)^m$.

\begin{prop}\label{Simp_Group_Convergence}
    Let $G_\bullet$ be a $\pi$-finite simplicial $p$-group, the bar construction $(BG)_\bullet$ is an $\ell$-adically continuous resolution of $BG$.
\end{prop}

\begin{proof}
    Consider the sequences $f(n) := \log_p|G_n|$ and $\cl{f}(n) := \log_p|N_n G|$. The iterated semidirect product decomposition of $G_n$ in terms of the $N_kG$-s (see, e.g., \cite[Proposition 12.]{mutlu2001iterated}), implies that the sequence $f(n)$ is the binomial transform of the sequence $(\cl{f}(n))$.
    By assumption,  there exists $d\ge 0$ such that $\cl{f}(n) = 0$ for $n>d$. Hence,
    \[
        f(n) = \sum_{k=0}^d \binom{n}{k}\cl{f}(k).
    \]
    is a polynomial. Consequently, the function $|G_n| = p^{f(n)}$ is $\ell$-adically continuous with
    \[
        p^{f(-1)} = \prod_{n=1}^\infty |N_nG|^{(-1)^n} = |G|.
    \]
    Passing to the bar construction, we get
    \[
        |(BG)_n| = |G_n|^n = p^{n f(n)},
    \]
    which is again $\ell$-adically continuous with extrapolation
    \[
        p^{-f(-1)} = |G|^{-1} = |BG|.
    \]
\end{proof}

\begin{rem}
    There is another standard simplicial model for the classifying space of a simplicial group $G_\bullet$, which is denoted by $(\cl{W}G)_\bullet$. It differs from $(BG)_\bullet$ and in particular has
    \[
        (\cl{W}G)_n = G_0 \times G_1 \times \cdots \times G_{n-1}.
    \]
    One easily checks that for a $\pi$-finite simplicial $p$-group $G_\bullet$, the simplicial set $(\cl{W}G)_\bullet$ is also an $\ell$-adically continuous resolution of the space $BG$. 
\end{rem}

\bibliographystyle{alpha}
\bibliography{references}
\addcontentsline{toc}{section}{References}

\end{document}